\documentclass[12pt,twoside]{amsart}
\usepackage{amsmath, amsthm, amscd, amsfonts, amssymb, graphicx, color}
\usepackage[all,knot,arc,import,poly]{xy}

\newtheorem{theorem}{Theorem} 

\newtheorem{rem}{\bf{Remark}}

\numberwithin{equation}{section}

\newtheorem{corollary}{Corollary}
\newtheorem{proposition}{Proposition}

\newtheorem{example}{Example}

\def\pn{\par\noindent}

\begin{document}

\title{The n-th smallest term for any finite sequence of real numbers.}
\author{Josimar da Silva Rocha}

\thanks{{\scriptsize
\hskip -0.4 true cm MSC(2010): Primary: 00A05; Secondary: 06A75.
}}
\maketitle

\begin{center}
\end{center}

\begin{abstract} 
In this paper we find the formula that gives the $n^{th}$ smallest term in a given finite sequence $\{x_{k}\}_{k=1}^{N}$ of real numbers.
\end{abstract}

\vskip 0.2 true cm



\bigskip
\bigskip


 In the literature, we can find many algorithms to ordination, such that Quicksort, Shellsort, Buble sort, Heapsort, Merge Sort and others.  Any these algorithms  can be founded in \cite{josd}. 
 
It is easy to see that these algorithms can be changed to find the n-th smallest term in a finite sequence of real numbers. However, in the recent  literature, we hasn't found  a formula that gives us the n-th smallest term for any  finite sequence of real numbers.
 
The purpose of this paper will be  to define a function that  gives us  the  n-th smallest term for any finite sequence of real numbers with N terms.

In order to introduce the notation, if $\{x_{k}\}_{k=1}^{N} $ is a finite sequence of real numbers with $N$ terms, then we denote by  $\{x^{(j)}_{k}\}_{k=1}^{N-1} $   the subsequence of $\{ x_{k}\}_{k=1}^{N}$ obtening by elimination of  $j$-th term of  $\{x_{k}\}_{k=1}^{N},$ that is 
\[ x^{(j)}_{k} = \left\{ \begin{array}{ll}
x_{k}, & \text{if  $ k < j$} \\
x_{k+1}, & \text{if  $k \geq j$ and $k<N$} \end{array} \right. \]

In general,
\[  x^{(j, t)}_{k} = (x_{k}^{(j)})^{(t)} \]
and
\[ x^{(j_{1}, j_{2}, \cdots, j_{t})}_{k} = \left(x_{k}^{(j_{1}, \cdots, j_{t-1})}\right)^{(j_{t})}. \]

\begin{proposition} Let  $\{x_{k}\}_{k=1}^{N} $ be a finite sequence of real numbers with $N$ terms. Let $\sigma $ be a permutation on $\{1, 2, \cdots, N\}$ with
\[ x_{1^{\sigma}} \leq x_{2^{\sigma}} \leq \cdots \leq x_{N^{\sigma}}. \] If $n$ is a positive integer such that $2 \leq n \leq N, $ then there is $j \in \{1, 2, \ldots, N - n + 2\}$ such that $x_{n^{\sigma}} \in \{ x_{k}^{(j)}\}_{k=1}^{N-1}. $ 
\end{proposition} 

\begin{proof} In fact, if  $A = \{1, \cdots, N - n + 2\}$  and 
$B = \{1^{\sigma}, 2^{\sigma}, \ldots, (n-1)^{\sigma}\} , $ by Inclusion-Exclusion Principle, we have $|A \cap B| = |A| + |B| - |A \cup B| =  (N- n + 2) + (n - 1) - |A \cup B| =  N + 1 - |A \cup B| \geq  N + 1 - N = 1.$  Consequently, $A \cup B \neq \varnothing $ and  there is   $ j \in A \cap B$ such that  $x_{n^{\sigma}} \in \{x_{k}^{(j)}\}_{k=1}^{N-1}.$ 
\end{proof}

\begin{rem} The function that affords us  the greatest element in a finite sequence of real numbers is given by following recursive formula:
\[ \max\{ x_{1}, x_{2}\} = \frac{x_{1} + x_{2} + |x_{1} - x_{2}|}{2} \]
\[ \max\{ x_{1}, x_{2}, \cdots, x_{N}\}  = \max\{ \max\{ x_{1}, x_{2}, \cdots, x_{N-1}\} , x_{N}\}   \] 

This function can be changed to calculate  the smallest element  in a finite sequence of real numbers by the following recursive formula
\[ \min\{ x_{1}, x_{2}\}  = -\max\{-x_{1}, -x_{2}\}  = \frac{x_{1} + x_{2} - |x_{1} - x_{2}|}{2} \]
\[ \min\{x_{1}, \cdots, x_{n}\}  = \min\{\min\{x_{1},\cdots, x_{N-1}\}, x_{N}\} \] 

These functions was used to demonstrate the {\em Stone-Weierstrass Theorem}  in   \cite{josa,josb,josc}.

\end{rem}

\begin{example} For three and four terms, we have
\[  \begin{array}{ll} &  \min\{x_{1}, x_{2}, x_{3}\} \\
= &  \min\{\min\{x_{1}, x_{2}\}, x_{3}\} \\ \\
= &  \displaystyle \frac{\min\{x_{1}, x_{2}\} + x_{3} - \mid \min\{x_{1}, x_{2}\} - x_{3} \mid }{2} \\ \\
= & \displaystyle  \frac{ \frac{x_{1} + x_{2} - \mid x_{1} - x_{2} \mid }{2} + x_{3} - \left|  \frac{x_{1} + x_{2} - \mid x_{1} - x_{2} \mid }{2}  - x_{3} \right|}{2} \\ \\
= &\displaystyle  \frac{x_{1} + x_{2} + 2x_{3} - \mid x_{1} - x_{2} \mid   - \left|  x_{1} + x_{2} - 2x_{3} - \mid x_{1} - x_{2} \mid   \right|}{4} 
\end{array}   \]

\[  \begin{array}{ll} &  \min\{x_{1}, x_{2}, x_{3}, x_{4}\} \\
= &  \min\{\min(x_{1}, x_{2}, x_{3}\}, x_{4}\}  \\ \\
= &  \displaystyle \frac{ \frac{x_{1} + x_{2} + 2x_{3} - \mid x_{1} - x_{2} \mid   - \left|  x_{1} + x_{2} - 2x_{3} - \mid x_{1} - x_{2} \mid   \right|}{4} + x_{4} - \left|  \frac{x_{1} + x_{2} + 2x_{3} - \mid x_{1} - x_{2} \mid   - \left|  x_{1} + x_{2} - 2x_{3} - \mid x_{1} - x_{2} \mid   \right|}{4}  - x_{4} \right|}{2}  \\ \\
= &  \frac{ x_{1} + x_{2} + 2x_{3} + 4x_{4}- \mid x_{1} - x_{2} \mid   - \left|  x_{1} + x_{2} - 2x_{3} - \mid x_{1} - x_{2} \mid   \right| - \left|  x_{1} + x_{2} + 2x_{3} - 4x_{4} - \mid x_{1} - x_{2} \mid   - \left|  x_{1} + x_{2} - 2x_{3} - \mid x_{1} - x_{2} \mid   \right|  \right|}{8}  \\
\end{array}   \]
\end{example}

\vspace{1cm}

\begin{theorem} Let $\{x_{k}\}_{k=1}^{N}$ be a finite sequence of real numbers with $N$ terms and let a positive integer $ n $ such that $ n \leq N, $ then
\[ T(n, \{x_{k}\}_{k=1}^{N}) = \left\{ \begin{array}{ll} \displaystyle   \max \bigcup_{j=1}^{N-n+2} \left\{ T\left(n - 1, \{x_{k}^{(j)}\}_{k=1}^{N-1}\right) \right\} , & \text{ if $n \geq 2$} \\
\min\{x_{1}, \cdots, x_{N} \}, & \text{ if $n = 1$} \end{array}\right. \]
satisfies $T(n, \{x_{k}\}_{k=1}^{N}) = x_{n^{\sigma}}, $ where $\sigma $ is a permutation on $\{1, 2, \cdots, N\} $ such that $x_{1^{\sigma}} \leq x_{2^{\sigma}} \leq \cdots \leq x_{N^{\sigma}}. $
\end{theorem}

\begin{proof} If $n = 1,$ then $T(1, \{x_{n}\}_{k=1}^{N}) = \min\{ x_{1}, \cdots, x_{n}\} = x_{1^{\sigma}}.$
If $N = 1, $ then $T(1, \{x_{n}\}_{k=1}^{1}) = x_{1} = x_{1^{\sigma}}. $ 

If $N = 2, $ then $T(1, \{x_{n}\}_{k=1}^{2}) = x_{1^{\sigma}}$ and $T(2, \{x_{n}\}_{k=1}^{2}) = \max\{ x_{2}, x_{1}\} = x_{2^{\sigma}}. $

Suppose, by induction on $n$ and $N$ that 
\[ T(n - 1, \{x_{k}^{(j)}\}_{k=1}^{N-1} ) = \left\{ \begin{array}{ll} x_{(n-1)^{\sigma}}, & \text{ if $j \not\in \{ 1^{\sigma}, 2^{\sigma}, \cdots, (n-1)^{\sigma}\}$ } \\
 x_{(n)^{\sigma}}, & \text{ if $j \in \{ 1^{\sigma}, 2^{\sigma}, \cdots, (n-1)^{\sigma}\}$ }, \end{array} \right. \]

As 
\[ T(n, \{x_{k}\}_{k=1}^{N}) = \max \bigcup_{j=1}^{N-n+2} \{ T(n-1, \{x_{k}^{(j)}\}_{k=1}^{N-1} ) \}, \] by Proposition 3, we have
\[ T(n, \{x_{k}\}_{k=1}^{N}) =  \max \bigcup_{j=1}^{N-n+2} \{ T(n-1, \{x_{k}^{(j)}\}_{k=1}^{N-1} ) \} = \max\{ x_{(n-1)^{\sigma}}, x_{n^{\sigma}}\} = x_{n^{\sigma}}. \]
\end{proof}

In Statistics, we can use the following Corollary to calculate the Mediane for any finite sequence of real numbers:

\begin{corollary}

The formula above affords us  to calculate the Mediane  $Md$ for any finite sequence  $\{x_{n}\}_{k=1}^{N}$ of real numbers  with $N$ terms:

\[ Md = \left\{ \begin{array}{ll}
T\left(\frac{N+1}{2}, \{x_{n}\}_{k=1}^{N}\right), & \text{if $n$  is  odd} \\ \\
\frac{T\left(\frac{N}{2}, \{x_{n}\}_{k=1}^{N}\right) + T\left(\frac{N}{2} + 1, \{x_{n}\}_{k=1}^{N}\right)}{2}, & \text{if  $n $ is even} \end{array}\right.  \]

\end{corollary} 

\begin{example}  The following formulas affords us the n-th  smallest term in a finite sequence of real numbers for $N \in \{1, 2, 3, 4\}$:

For $N = 1:$ 
\[ T(1, \{x_{k}\}_{k=1}^{1}) = x_{1} \]

For $N = 2:$ 
\[ T(1, \{x_{k}\}_{k=1}^{2}) = \min\{ x_{1}, x_{2}\} \]
\[ T(2, \{x_{k}\}_{k=1}^{2}) = \max\{ x_{2}, x_{1}\} \]

For $N = 3:$ 

\[ T(1, \{x_{k}\}_{k=1}^{3}) = \min\{ x_{1}, x_{2}, x_{3}\} \]
\[ T(2, \{x_{k}\}_{k=1}^{3}) = \max\left\{ \min\{ x_{2}, x_{3}\}, \min\{ x_{1}, x_{3}\}, \min\{ x_{1}, x_{2}\} \right\} \]
\[ T(3, \{x_{k}\}_{k=1}^{3}) = \max\left\{ \max\{ x_{3}, x_{2}\}, \max\{ x_{3}, x_{1}\} \right\} \]

For $N = 4:$ 

\[ T(1, \{x_{k}\}_{k=1}^{4}) = \min\{ x_{1}, x_{2}, x_{3}, x_{4}\}  \]
\[ T(2, \{x_{k}\}_{k=1}^{4}) = \max\left\{ \min\{ x_{2}, x_{3}, x_{4}\}, \min\{ x_{1}, x_{3}, x_{4}\}, \min\{ x_{1}, x_{2}, x_{3}\} \right\} \]
\[ T(3, \{x_{k}\}_{k=1}^{4}) = \max\left\{ \begin{array}{l}  \max\{ \min\{x_{3}, x_{4}\}, \min\{ x_{2}, x_{4}\}, \min\{x_{2}, x_{3}\} \}, \\
 \max\{ \min\{x_{3}, x_{4}\}, \min\{x_{1}, x_{4}\}, \min\{x_{1}, x_{3}\}\}, \\
\max\{ \min\{x_{2}, x_{4}\}, \min\{x_{1}, x_{4}\}, \min\{x_{1}, x_{2}\} \} \end{array}  \right\} \]
\[ T(4, \{x_{k}\}_{k=1}^{4}) = \max\left\{ \max\{\max\{x_{4}, x_{3}\}, \max\{x_{4}, x_{2}\} \} ,\max\{ \max\{ x_{4}, x_{3}\}, \max\{x_{4}, x_{1}\}\} \right\}  \]
\end{example}

\bigskip
\bigskip

{\footnotesize \pn{\bf Josimar da Silva Rocha}\; \\ {Coordination of
Mathematics (COMAT)}, {Universidade Tecnol\'ogica Federal do Paran\'a, 86300-000,} {Corn\'elio Proc\'opio-PR, Brazil}\\
{\tt Email: jsrocha74@gmail.com}\\


\begin{thebibliography}{99}
\bibitem{josa} Robert G. Barthe,  \emph{The Elements of Real Analysis,} John Wiley \& Sons, New York, 1964.
\bibitem{josb} Serge Lang, \emph{Analysis I,} Addison - Wesley, Reading, Mass.,  1968.
\bibitem{josd} Jeffrey J. McConnell, {\em Analysis of Algorithms: An Active Learning Approach,} Jones and Bartlett, Ontario, 2008.
\bibitem{josc}  Walter Rudin, \emph{ Principles of Mathematical Analysis, } McGraw-Hill, New York, 1976.


\end{thebibliography}
\end{document}